\documentclass[11pt,british]{amsart}

\usepackage{babel,eucal,url,amssymb,
enumerate,amscd,
}

\usepackage[pagebackref]{hyperref}

\theoremstyle{plain}
\newtheorem{thm}{Theorem}[section]
\newtheorem{lem}[thm]{Lemma}
\newtheorem{pro}[thm]{Proposition}
\newtheorem{co}[thm]{Corollary}

\theoremstyle{definition}
\newtheorem{defn}[thm]{Definition}

\theoremstyle{remark}

\newtheorem{example}[thm]{Example}

\newcommand{\Gtwo}{\ifmmode{{\rm G}_2}\else{${\rm G}_2$}\fi}

\date{\today}

\begin{document}

\title[]
 {On quasi-para-Sasakian manifolds}

\author[S. Zamkovoy]{Simeon Zamkovoy}

\address{
University of Sofia "St. Kl. Ohridski"\\
Faculty of Mathematics and Informatics\\
Blvd. James Bourchier 5\\
1164 Sofia, Bulgaria}
\email{zamkovoy@fmi.uni-sofia.bg}






\begin{abstract}
In this paper we study quasi-para-Sasakian manifolds.
We characterize these manifolds by tensor equations and study their properties. We are devoted to the study of $\eta-$Einstein manifolds.
We show that a conformally flat quasi-para-Sasakian manifold is a space of constant negative curvature $-1$
and we prove that if a quasi-para-Sasakian manifold is a space of constant $\varphi-$para-holomorphic sectional curvature $H=-1$,
then it is a space of constant curvature. Finally, the object of the present paper is to study a 3-dimensional quasi-para-Sasakian manifold, satisfying certain curvature conditions.
Among other, it is proved that any 3-dimensional quasi-para-Sasakian manifold with $\eta-$parallel Ricci tensor is of constant scalar curvature.
\end{abstract}

\newcommand{\g}{\mathfrak{g}}
\newcommand{\s}{\mathfrak{S}}
\newcommand{\D}{\mathcal{D}}
\newcommand{\F}{\mathcal{F}}
\newcommand{\R}{\mathbb{R}}
\newcommand{\K}{\mathbb{K}}
\newcommand{\U}{\mathbb{U}}
\newcommand{\diag}{\mathrm{diag}}
\newcommand{\End}{\mathrm{End}}
\newcommand{\im}{\mathrm{Im}}
\newcommand{\id}{\mathrm{id}}
\newcommand{\Hom}{\mathrm{Hom}}

\newcommand{\Rad}{\mathrm{Rad}}
\newcommand{\rank}{\mathrm{rank}}
\newcommand{\const}{\mathrm{const}}
\newcommand{\tr}{{\rm tr}}
\newcommand{\ltr}{\mathrm{ltr}}
\newcommand{\codim}{\mathrm{codim}}
\newcommand{\Ker}{\mathrm{Ker}}

\newcommand{\thmref}[1]{Theorem~\ref{#1}}
\newcommand{\propref}[1]{Proposition~\ref{#1}}
\newcommand{\corref}[1]{Corollary~\ref{#1}}
\newcommand{\secref}[1]{\S\ref{#1}}
\newcommand{\lemref}[1]{Lemma~\ref{#1}}
\newcommand{\dfnref}[1]{Definition~\ref{#1}}


\newcommand{\ee}{\end{equation}}
\newcommand{\be}[1]{\begin{equation}\label{#1}}

\maketitle

\section{Introduction}\label{sec-1}
In this paper we study a class of paracontact pseudo-Riemannian manifolds satisfying some special conditions. These manifolds are
analogues to the quasi-Sasakian manifolds and they belong of the class $\mathbb{G}_5$ of the classification given in \cite{ZN}.
We characterize these manifolds by tensor equations and study their properties. From the definition
by means of the tensor equations, it is easily verified that the structure is normal, but not para-Sasakian.
We are devoted to the study of $\eta-$Einstein manifolds. We show that a conformally flat quasi-para-Sasakian manifold is a space of constant negative curvature $-1$
and we prove that if a quasi-para-Sasakian manifold is a space of constant $\varphi-$para-holomorphic sectional curvature $H=-1$,
then it is a space of constant curvature. In the last section, we study the 3-dimensional quasi-para-Sasakian manifolds. We prove that any 3-dimensional quasi-para-Sasakian manifold
satisfying the condition $R(X,Y).Ric=0$ is a manifold of constant negative curvature, where $R(X,Y)$ is considered as a derivation
of the tensor algebra at each point of manifold ($X,Y$ are tangent vectors). We study locally $\varphi-$symmetric quasi-para-Sasakian
manifolds and obtain a necessary and sufficient condition a 3-dimensional quasi-para-Sasakian manifold to be locally $\varphi-$symmetric.
We obtain some interesting results about a 3-dimensional quasi-para-Sasakian manifolds with $\eta-$parallel Ricci tensor.
We give an example for a 3-dimensional quasi-para-Sasakian manifold with a scalar curvature equal to $-6$.
\section{Preliminaries}\label{sec-2}
A (2n+1)-dimensional smooth manifold $M^{(2n+1)}$
has an \emph{almost paracontact structure}
$(\varphi,\xi,\eta)$
if it admits a tensor field
$\varphi$ of type $(1,1)$, a vector field $\xi$ and a 1-form
$\eta$ satisfying the  following compatibility conditions 
\begin{eqnarray}
  \label{f82}
    & &
    \begin{array}{cl}
     (i)   & \varphi(\xi)=0,\quad \eta \circ \varphi=0,\quad
     \\[5pt]
     (ii)  & \eta (\xi)=1 \quad \varphi^2 = id - \eta \otimes \xi,
     \\[5pt]
     (iii) & \textrm{distribution $\mathbb {D}: p \in M \longrightarrow \mathbb {D}_p\subset T_pM:$}
     \\[1pt]
     & \textrm{$\mathbb D_p=Ker \eta=\{X\in T_pM: \eta (X)=0\}$ is called {\it paracontact}}
     \\[1pt]
     & \textrm{{\it distribution} generated by $\eta$.}
    \end{array}
\end{eqnarray}

The tensor field $\varphi $ induces an almost paracomplex structure \cite{KW} on each
fibre on $\mathbb D$ and $(\mathbb D, \varphi , g_{\vert \mathbb D})$ is a $2n$-dimensional
almost paracomplex manifold. Since $g$ is non-degenerate metric on $M$ and $\xi $ is non-isotropic,
the paracontact distribution $\mathbb D$ is non-degenerate.

An immediate consequence of the definition of the almost
paracontact structure is that the endomorphism $\varphi$ has rank
$2n$,
$\varphi \xi=0$ and $\eta \circ \varphi=0$, (see \cite{B1,B2}
for the almost contact case).


If a manifold $M^{(2n+1)}$ with $(\varphi,\xi,\eta)$-structure
admits a pseudo-Riemannian metric $g$ such that
\begin{equation}\label{con}
g(\varphi X,\varphi Y)=-g(X,Y)+\eta (X)\eta (Y),
\end{equation}
then we say that $M^{(2n+1)}$ has an almost paracontact metric structure and
$g$ is called \emph{compatible}. Any compatible metric $g$ with a given almost paracontact
structure is necessarily of signature $(n+1,n)$.

Note that setting $Y=\xi$, we have
$\eta(X)=g(X,\xi).$

Further, any almost paracontact structure admits a compatible metric.
\begin{defn}
If $g(X,\varphi Y)=d\eta(X,Y)$ (where
$d\eta(X,Y)=\frac12(X\eta(Y)-Y\eta(X)-\eta([X,Y])$ then $\eta$ is
a paracontact form and the almost paracontact metric manifold
$(M,\varphi,\eta,\xi,g)$ is said to be a $\emph{paracontact metric
manifold}$.
\end{defn}
A paracontact metric manifold for which $\xi$ is Killing is called a $K-\emph{paracontact}$ $\emph{manifold}$. A paracontact structure on $M^{(2n+1)}$ naturally gives rise to an almost paracomplex structure on the product $M^{(2n+1)}\times\Re$. If this almost paracomplex structure is integrable, then the given paracontact metric manifold is said to be a $\emph{para-Sasakian manifold}$. Equivalently, (see \cite{Z}) a paracontact metric manifold is a para-Sasakian manifold if and only if
\begin{equation}\label{8}
(\nabla_{X}\varphi)Y=-g(X,Y)\xi+\eta(Y)X,
\end{equation}
 for all vector fields $X$ and $Y$ (where
$\nabla$ is the Livi-Civita connection of $g$).
\begin{defn}\label{d1}
If $(\nabla_X\varphi)Y=g(X,Y)\xi-\eta(Y)X,$ then the manifold
$(M,\varphi,\eta,\xi,g)$ is said to be a $\emph{quasi-para-Sasakian manifold}$.
\end{defn}
From $Definition~\ref{d1}$ (see \cite{JW}) we have
\begin{equation}\label{z1}
\nabla_{X}\xi=\varphi X.
\end{equation}

In \cite{ZN}, it is proved that $(M,\varphi,\eta,\xi,g)$ is normal, since $\xi$ is a Killing vector field and the manifold is not para-Sasakian. Thus we have
\begin{pro}
Let $(M,\varphi,\eta,\xi,g)$ be a quasi-para-Sasakian manifold. Then $(M,\varphi,\eta,\xi,g)$ is normal but not para-Sasakian.
\end{pro}
Denoting by $\pounds$ the Lie differentiation of $g$, we see
\begin{pro}\label{pro2}
Let $(M,\varphi,\eta,\xi,g)$ be a quasi-para-Sasakian manifold. Then we have
\begin{equation}\label{e1}
(\nabla_X\eta)Y=-g(X,\varphi Y),
\end{equation}
\begin{equation}\label{e2}
\pounds_{\xi}g=0,
\end{equation}
\begin{equation}\label{e3}
\pounds_{\xi}\varphi=0,
\end{equation}
\begin{equation}\label{e4}
\pounds_{\xi}\eta=0,
\end{equation}
\begin{equation}\label{e4.1}
d\eta(X,Y)=-g(X,\varphi Y),
\end{equation}
where $X,Y\in T_pM.$
\end{pro}
Since the proof of $Proposition~\ref{pro2}$ follows by routine calculation, we shall omit it.

Denoting by $R$ the curvature tensor of $\nabla$, we have the following
\begin{defn}
An almost paracontact structure $(\varphi,\xi,\eta,g)$ is said to be $\emph{locally}$ $\emph{symmetric}$ if $(\nabla_WR)(X,Y,Z)=0$, for all vector
fields $W,X,Y,Z\in T_pM$ .
\end{defn}
\begin{defn}
An almost paracontact structure $(\varphi,\xi,\eta,g)$ is said to be $\emph{locally $\varphi-$}$ $\emph{symmetric}$ if $\varphi^2(\nabla_WR)(X,Y,Z)=0$, for all vector
fields $W,X,Y,Z$ orthogonal to $\xi$.
\end{defn}

Finally, the sectional curvature $K(\xi,X)=\epsilon_XR(X,\xi,\xi,X)$, where $|X|=\epsilon_X=\pm 1$, of a plane section spanned by $\xi$ and the vector $X$ orthogonal to $\xi$ is called $\emph{$\xi$-sectional curvature}$, whereas the sectional curvature $K(X,\varphi X)=-R(X,\varphi X,\varphi X,X)$, where $|X|=-|\varphi X|=\pm 1$,
of a plane section spanned by vectors $X$ and $\varphi X$ orthogonal to $\xi$ is called a $\emph{$\varphi$-para-holomorphic sectional curvature}$.
\section{Some properties of quasi-para-Sasakian manifolds}\label{sec-3}
The following result is well-known from the theory of para-Sasakian manifolds: $K(X,\xi)=-1$ and if a para-Sasakian manifold is locally symmetric, then it is of constant negative curvature $-1$(\cite{Z1}).
For quasi-para-Sasakian manifolds we get
\begin{pro}\label{pro3}
Let $(M,\varphi,\eta,\xi,g)$ be a quasi-para-Sasakian manifold. Then we have
\begin{equation}\label{e5}
R(X,Y)\xi=\eta(X)Y-\eta(Y)X,
\end{equation}
\begin{equation}\label{e5.1}
R(X,\xi)Y=g(X,Y)\xi-\eta(Y)X,
\end{equation}
\begin{equation}\label{e5.2}
Ric(X,\xi)=-2n\eta(X),
\end{equation}
\begin{equation}\label{e6}
K(X,\xi)=-1,
\end{equation}
\begin{equation}\label{e7}
(\nabla_ZR)(X,Y,\xi)=-R(X,Y)\varphi Z+g(X,\varphi Z)Y-g(Y,\varphi Z)X,
\end{equation}
where $Ric$ is the Ricci tensor and $X,Y,Z\in T_pM.$
\end{pro}
\begin{proof}
The equation \eqref{e5} follows directly from \eqref{z1}, \eqref{e1} and the definition of the curvature $R$. The equations  \eqref{e5.1}, \eqref{e5.2}  and \eqref{e6} are a consequence of \eqref{e5}. By virtue of \eqref{z1}
\eqref{e1} and \eqref{e5} we get
\eqref{e7}:
$$(\nabla_ZR)(X,Y,\xi)=Z(R(X,Y)\xi)-R(\nabla_ZX,Y)\xi-R(X,\nabla_ZY)\xi-R(X,Y)\nabla_Z\xi=$$
$$=-R(X,Y)\varphi Z+g(X,\varphi Z)Y-g(Y,\varphi Z)X.$$
\end{proof}
\begin{co}\label{co1}
If $(M,\varphi,\eta,\xi,g)$ is locally symmetric, then it is of constant negative curvature $-1$.
\end{co}
\begin{proof}
$Corollary~\ref{co1}$ follows from \eqref{e7}.
\end{proof}

\section{$\eta-$Einstein manifolds}\label{sec-4}
An almost paracontact pseudo-Riemannian manifold is called $\eta-\emph{Einstein}$, if the Ricci tensor $Ric$ satisfies $Ric=a.g+b.\eta\otimes\eta,$ where $a,b$ are smooth scalar functions on $M$. If a para-Sasakian manifold is $\eta-$Einstein and $n>1$, then $a$ and $b$ are constant (see \cite{Z}).
\begin{pro}
Let $(M,\varphi,\eta,\xi,g)$ be a quasi-para-Sasakian manifold. If $M$ is an $\eta-$Einstein manifold, we have
\begin{equation}\label{e8}
a+b=-2n,
\end{equation}
\begin{equation}\label{e9}
a=constant\quad and \quad b=constant,\quad n>1.
\end{equation}
\end{pro}
\begin{proof}
The equation \eqref{e8} follows from $Ric(X,\xi)=-2n\eta(X)$ which is derived from \eqref{e5}. As $M$ is an $\eta-$Einstein manifold, the scalar curvature $scal$ is equal to $2n(a-1)$. We define the Ricci operator $Q$ as follows:
$g(QX,Y)=Ric(X,Y)$. By the identity $Y(scal)=2nY(a)$ and the trace of the map $[X\rightarrow (\nabla_XQ)Y]$, we have $\xi(scal)=0,\quad \xi(a)=\xi(b)=0$. From the identity
$$(\nabla_ZRic)(X,Y)=Z(a)g(X,Y)+Z(b)\eta(X)\eta(Y)-b\eta(Y)g(Z,\varphi X)-b\eta(X)g(Z,\varphi Y)$$
we have $Y(scal)=2Y(a)$, but $Y(scal)=2nY(a)$ and therefore $2(n-1)Y(a)=0$. For $n>1$ we obtain $a=constant$ and $b=constant$.
\end{proof}

\section{Curvature tensor}\label{sec-5}
Now we prove the following
\begin{pro}\label{pro4}
Let $(M,\varphi,\eta,\xi,g)$ be a quasi-para-Sasakian manifold. Then we have the following identities
\begin{equation}\label{e10}
R(X,Y)\varphi Z-\varphi R(X,Y)Z=g(Y,Z)\varphi X-g(X,Z)\varphi Y-g(Y,\varphi Z)X+,
\end{equation}
$$+g(X,\varphi Z)Y,$$
\begin{equation}\label{e11}
R(\varphi X,\varphi Y)Z=-R(X,Y)Z-g(Y,Z)X+g(X,Z)Y+
\end{equation}
$$+g(Y,\varphi Z)\varphi X-g(X,\varphi Z)\varphi Y,$$
where $X,Y,Z\in T_pM.$
\end{pro}
\begin{proof}
The equation \eqref{e10} follows from the Ricci's identity:
$$\nabla_X\nabla_Y\varphi-\nabla_Y\nabla_X\varphi-\nabla_{[X,Y]}\varphi=R(X,Y)\varphi Z-\varphi R(X,Y)Z.$$
We verify \eqref{e11}: By \eqref{e10}, we have
$$R(X,Y,\varphi Z,\varphi W)-g(\varphi R(X,Y)Z,\varphi W)=$$
$$=g(Y,Z)g(\varphi X,\varphi W)-g(X,Z)g(\varphi Y,\varphi W)-$$
$$-g(Y,\varphi Z)g(X,\varphi W)+g(X,\varphi Z)g(Y,\varphi W).$$
Using $\eta(R(X,Y)Z)=-\eta(X)g(Y,Z)+\eta(Y)g(X,Z)$, the above formula takes the form
$$R(\varphi Z,\varphi W,X,Y)=-R(Z,W,X,Y)-g(Y,Z)g(X,W)+g(X,Z)g(Y,W)-$$
$$-g(Y,\varphi Z)g(X,\varphi W)+g(X,\varphi Z)g(Y,\varphi W).$$
\end{proof}
As an application of $Proposition~\ref{pro4}$, we shall prove the following proposition.
\begin{pro}\label{pro5}
Let $(M,\varphi,\eta,\xi,g)$ be a quasi-para-Sasakian manifold of dimension greater than $3$. If $M$ is conformally flat, then $M$ is a space of constant negative curvature $-1$.
\end{pro}
\begin{proof}
Since $M$ is conformally flat, the curvature tensor of $M$ is written as
\begin{equation}\label{e12}
R(X,Y)Z=\frac{1}{2n-1}(Ric(Y,Z)X-Ric(X,Z)Y+g(Y,Z)QX-g(X,Z)QY)-
\end{equation}
$$+\frac{scal}{2n(2n-1)}(g(X,Z)Y-g(Y,Z)X).$$
We calculate $R(\xi,Y)\xi$ using the previous formula. Using \eqref{e5} and
$$Ric(X,\xi)=-2n\eta(X),$$
we get
\begin{equation}\label{e13}
2nRic(Y,Z)=(scal+2n)g(Y,Z)-(scal+4n^2+2n)\eta(Y)\eta(Z).
\end{equation}
By virtue of \eqref{e10}, \eqref{e12} and \eqref{e13}, we have
\begin{equation}\label{e14}
(scal+4n^2+2n)(g(Y,\varphi Z)X-g(X,\varphi Z)Y+g(X,Z)\varphi Y-g(Y,Z)\varphi X+
\end{equation}
$$+g(X,\varphi Z)\eta(Y)\xi-g(Y,\varphi Z)\eta(X)\xi+\eta(Y)\eta(Z)\varphi X-\eta(X)\eta(Z)\varphi Y)=0.$$
Let $(e_1,...,e_n,\varphi e_1,...,\varphi e_n,\xi)$ be an orthonormal basis of $T_pM$. Setting $X=e_1,Y=e_2$ and $Z=\varphi e_2$ in \eqref{e14}, we see $scal=-2n(2n+1)$. Thus we have $Ric=-2ng$.
$Proposition~\ref{pro4}$ follows from \eqref{e12}.
\end{proof}

In a para-Sasakian manifold with constant $\varphi-$para-holomorphic sectional curvature, say $H$, the curvature tensor has a special feature (see \cite{Z1}):
The necessary and sufficient condition for a para-Sasakian manifold to have constant $\varphi-$para-holomorphic sectional curvature $H$ is
$$4R(X,Y)Z=(H-3)(g(Y,Z)X-g(X,Z)Y)+(H+1)(\eta(X)\eta(Z)Y-\eta(Y)\eta(Z)X+$$
$$+\eta(Y)g(X,Z)\xi-\eta(X)g(Y,Z)\xi+g(Y,\varphi Z)\varphi X-g(X,\varphi Z)\varphi Y+2g(\varphi X,Y)\varphi Z).$$
In our case we have
\begin{pro}\label{pro6}
Let $(M,\varphi,\eta,\xi,g)$ be a quasi-para-Sasakian manifold. The necessary and sufficient condition for $M$ to have constant $\varphi-$para-holomorphic sectional curvature $H$  is
\begin{equation}\label{e16}
4R(X,Y)Z=(H-3)(g(Y,Z)X-g(X,Z)Y)+(H+1)(\eta(X)\eta(Z)Y-\eta(Y)\eta(Z)X+
\end{equation}
$$+\eta(Y)g(X,Z)\xi-\eta(X)g(Y,Z)\xi+g(Y,\varphi Z)\varphi X-g(X,\varphi Z)\varphi Y+2g(\varphi X,Y)\varphi Z),$$
where $X,Y,Z\in T_pM.$
\end{pro}
\begin{proof}
For any vector fields $X,Y\in\mathbb D$, we have
\begin{equation}\label{e17}
R(X,\varphi X,X,\varphi X)=Hg^2(X,X)
\end{equation}
By identity \eqref{e10} we get
\begin{equation}\label{e17}
R(X,\varphi Y,X,\varphi Y)=R(X,\varphi Y,Y,\varphi X)-g^2(X,\varphi Y)+g^2(X,Y)-g(X,X)g(Y,Y),
\end{equation}
\begin{equation}\label{e18}
R(X,\varphi X,Y,\varphi X)=R(X,\varphi X,X,\varphi Y). \end{equation}
Substituting $X+Y$ in \eqref{e10} and using the  Bianchi identity, we obtain \begin{equation}\label{e19}
2R(X,\varphi X,X,\varphi Y)+2R(Y,\varphi Y,Y,\varphi X)+3R(X,\varphi Y,Y,\varphi X)-R(X,Y,X,Y)=
\end{equation}
$$=H(2g^2(X,Y)+g(X,X)g(Y,Y)+2g(X,Y)g(X,X)+2g(X,Y)g(Y,Y)).$$
Replacing $Y$ by $-Y$ in \eqref{e19} and summing it to \eqref{e19} we have
\begin{equation}\label{e20}
3R(X,\varphi Y,Y,\varphi X)-R(X,Y,X,Y)=H(2g^2(X,Y)+g(X,X)g(Y,Y)).
\end{equation}
Replacing $Y$ by $\varphi Y$ in \eqref{e20} and from identities \eqref{e21}, \eqref{e17} and \eqref{e20}, we get
\begin{equation}\label{e21}
4R(X,Y,X,Y)=(H-3)(g^2(X,Y)-g(X,X)g(Y,Y))+(H+1)g^2(X,\varphi Y).
\end{equation}
Let $X,Y,Z,W\in\mathbb D$. Calculating $R(X+Z,Y+W,X+Z,Y+W)$ and using \eqref{e21} we see
\begin{equation}\label{e22}
4R(X,Y,Z,W)+4R(X,W,Z,Y)=(H-3)(g(X,Y)g(Z,W)+g(X,W)g(Y,Z)-
\end{equation}
$$-2g(X,Z)g(Y,W))+3(H+1)(g(X,\varphi Y)g(Z,\varphi W)+g(X,\varphi W)g(Z,\varphi Y))$$
and we have
\begin{equation}\label{e23}
-4R(X,Z,Y,W)-4R(X,W,Y,Z)=-(H-3)(g(X,Z)g(Y,W)+g(X,W)g(Y,Z)-
\end{equation}
$$-2g(X,Y)g(Z,W))-3(H+1)(g(X,\varphi Z)g(Y,\varphi W)+g(X,\varphi W)g(Y,\varphi Z)).$$
Adding \eqref{e22} to \eqref{e23} and using Bianchi identity, we get
\begin{equation}\label{e24}
4R(X,Y,Z,W)=(H-3)(g(X,W)g(Y,Z)-g(X,Z)g(Y,W))+
\end{equation}
$$+(H+1)(g(X,\varphi W)g(\varphi Y,Z)-g(X,\varphi Z)g(\varphi Y,W)+2g(X,\varphi Y)g(Z,\varphi W)).$$
For any vector fields  $X,Y,Z,W\in T_pM$ we have  $\varphi X,\varphi Y,\varphi Z,\varphi W\in\mathbb D$, and using \eqref{e24}, \eqref{e5},
\eqref{e10} and \eqref{e11}, we get \eqref{e16}.
\end{proof}

Recall that, if a para-Sasakian manifold has a constant $\varphi-$para-holomorphic section curvature, then it is $\eta$-Einstein. Similarly, we have the following theorem in our case:

\begin{thm}\label{t5}
Let $(M,\varphi,\eta,\xi,g)$ be a quasi-para-Sasakian manifold. If $M$ is a space of constant $\varphi-$para-holomorphic sectional curvature $H$, then $M$ is $\eta$-Einstein.
\end{thm}
\begin{proof}
By virtue of $Proposition~\ref{pro6}$, $M$ is an $\eta-$Einstein manifold and
$$Ric=\frac{1}{2}(n(H-3)+H+1)g-\frac{1}{2}(n+1)(H+1)\eta\otimes\eta.$$
\end{proof}
From the $Theorem~\ref{t5}$ we have the following
\begin{co}\label{co2}
Let $(M,\varphi,\eta,\xi,g)$ be a quasi-para-Sasakian manifold of a constant $\varphi-$para-holomorphic sectional curvature $H=-1$. Then $M$ is a space of constant curvature.
\end{co}
We can generalize $Corollary~\ref{co1}$ slightly as follows:
\begin{pro}\label{pro3}
Let $(M,\varphi,\eta,\xi,g)$ be a quasi-para-Sasakian manifold. If $M$ satisfies the Nomizu's condition, i.e., $R(X,Y).R=0$, for any $X,Y\in T_pM$, then it is of constant negative curvature $-1$.
\end{pro}
\begin{proof}
Let $X,Y\in \mathbb D$ and $g(X,Y)=0$. Then, using \eqref{e5} and \eqref{e5.1} above, we obtain
$$(R(X,\xi)R)(X,Y)Y=R(X,\xi)R(X,Y)Y-R(R(X,\xi)X,Y)Y-R(X,R(X,\xi)Y)Y-$$
$$-R(X,Y)R(X,\xi)Y=R(X,Y,Y,X)\xi-R(X,Y,Y,\xi)X-g(X,X)R(\xi,Y)Y=$$
$$=(R(X,Y,Y,X)+g(X,X)g(Y,Y))\xi.$$
From the identity $R(X,Y)R=0$, we get $R(X,Y,Y,X)=-g(X,X)g(Y,Y)$, which implies that $(M,\varphi,\eta,\xi,g)$ is of constant $\varphi-$para-holomorphic sectional curvature $-1$, and hence it is of constant curvature $-1$.
\end{proof}
The PC-Bochner curvature tensor on $M$ is defined by \cite{Z1}
$$\mathbf{B}(X,Y,Z,W)=R(X,Y,Z,W)+\frac{1}{2n+4}(Ric(X,Z)g(Y,W)-Ric(Y,Z)g(X,W)+$$
$$+Ric(Y,W)g(X,Z)-Ric(X,W)g(Y,Z)+Ric(\varphi X,Z)g(Y,\varphi W)-$$
$$-Ric(\varphi Y,Z)g(X,\varphi W)+Ric(\varphi Y,W)g(X,\varphi Z)-Ric(\varphi X,W)g(Y,\varphi Z)+$$
$$+2Ric(\varphi X,Y)g(Z,\varphi W)+2Ric(\varphi Z,W)g(X,\varphi Y)-Ric(X,Z)\eta(Y)\eta(W)+$$
$$+Ric(Y,Z)\eta(X)\eta(W)-Ric(Y,W)\eta(X)\eta(Z)+Ric(X,W)\eta(Y)\eta(Z))+$$
$$+\frac{k-4}{2n+4}(g(X,Z)g(Y,W)-g(Y,Z)g(X,W))-\frac{k+2n}{2n+4}(g(Y,\varphi W)g(X,\varphi Z)-$$
$$-g(X,\varphi W)g(Y,\varphi Z)+2g(X,\varphi Y)g(Z,\varphi W))-\frac{k}{2n+4}(g(X,Z)\eta(Y)\eta(W)-$$
$$-g(Y,Z)\eta(X)\eta(W)+g(Y,W)\eta(X)\eta(Z)-g(X,W)\eta(Y)\eta(Z)),$$
where $k=-\frac{scal-2n}{2n+2}$.

Using the PC-Bochner curvature tensor we have
\begin{thm}\label{t6}
Let $(M,\varphi,\eta,\xi,g)$ be a quasi-para-Sasakian manifold. Then $M$ is with constant $\varphi-$para-holomorphic sectional curvature if and only if it is $\eta$-Einstein and with vanishing PC-Bochner curvature tensor.
\end{thm}
Since the proof of this Theorem is done using the same method as in the proof of the Theorem 5.5 in \cite{Z1}, we shall omit it.
\section{3-dimensional quasi-para-Sasakian manifolds}\label{sec-6}
In a 3-dimensional pseudo-Riemannian manifold, we have
\begin{equation}\label{e25}
R(X,Y)Z=g(Y,Z)QX-g(X,Z)QY+g(QY,Z)X-g(QX,Z)Y-
\end{equation}
$$-\frac{scal}{2}(g(Y,Z)X-g(X,Z)Y).$$
Setting $Z=\xi$ in \eqref{e24} and using \eqref{e5} and \eqref{e5.2}, we have
\begin{equation}\label{e26}
\eta(Y)QX-\eta(X)QY=(\frac{scal}{2}+1)(\eta(Y)X-\eta(X)Y).
\end{equation}
Setting $Y=\xi$ in \eqref{e26} and then using \eqref{e5.2} (for n=1), we get
$$QX=\frac{1}{2}[(scal+2)X-(scal+6)\eta(X)\xi]$$
i.e.,
\begin{equation}\label{e27}
Ric(Y,Z)=\frac{(scal+2)}{2}g(Y,Z)-\frac{(scal+6)}{2}\eta(Y)\eta(Z).
\end{equation}
\begin{lem}\label{l1}
A 3-dimensional quasi-para-Sasakian manifold is a manifold of constant negative curvature if and only if the scalar curvature $scal=-6$.
\end{lem}
\begin{proof}
Using \eqref{e27} in \eqref{e25}, we get
\begin{equation}\label{e28}
R(X,Y)Z=\frac{(scal+4)}{2}(g(Y,Z)X-g(X,Z)Y)-
\end{equation}
$$-\frac{(scal+6)}{2}(g(Y,Z)\eta(X)\xi-g(X,Z)\eta(Y)\xi+\eta(Y)\eta(Z)X-\eta(X)\eta(Z)Y)$$
and now the Lemma is obvious.
\end{proof}
Let us consider a 3-dimensional quasi-para-Sasakian manifold which satisfies the  condition
\begin{equation}\label{e29}
R(X,Y).Ric=0,
\end{equation}
for any vector fields $X,Y\in T_pM$.
\begin{thm}\label{t1}
A 3-dimensional quasi-para-Sasakian manifold $(M,\varphi,\eta,\xi,g)$ satisfying the condition $R(X,Y).Ric=0$ is a manifold of constant negative curvature $-1$.
\end{thm}
\begin{proof}
From \eqref{e29}, we have
\begin{equation}\label{e30}
Ric(R(X,Y)U,V)+Ric(U,R(X,Y)V)=0.
\end{equation}
Setting $X=\xi$ and using \eqref{e5.1}, we obtain
\begin{equation}\label{e31}
\eta(U)Ric(Y,V)-g(Y,U)Ric(\xi,V)+\eta(V)Ric(U,\xi)-g(Y,V)Ric(\xi,U)=0.
\end{equation}
Using \eqref{e5.2} in \eqref{e31}, we have
\begin{equation}\label{e32}
\eta(U)Ric(Y,V)+2g(Y,U)\eta(V)+\eta(V)Ric(Y,U)+2g(Y,V)\eta(U)=0.
\end{equation}
Taking the trace in \eqref{e32}, we get
\begin{equation}\label{e33}
Ric(\xi,V)+8\eta(V)+scal\eta(V)=0.
\end{equation}
Using \eqref{e5.2} in \eqref{e33}, we obtain
$$(scal+6)\eta(V)=0.$$
This gives $scal=-6$ (since $\eta(V)\neq 0$), which implies, by $Lemma~\ref{l1}$, that the manifold is of constant negative curvature $-1$.
\end{proof}

\begin{thm}\label{t2}
A 3-dimensional quasi-para-Sasakian manifold $(M,\varphi,\eta,\xi,g)$ is locally $\varphi-$symmetric if and only if  the scalar curvature $scal$ is constant.
\end{thm}
\begin{proof}
Differentiating \eqref{e28} covariantly with respect to $W$ we get
\begin{equation}\label{e45}
(\nabla_WR)(X,Y,Z)=\frac{W(scal)}{2}(g(Y,Z)X-g(X,Z)Y)-
\end{equation}
$$-\frac{W(scal)}{2}(g(Y,Z)\eta(X)\xi-g(X,Z)\eta(Y)\xi+\eta(Y)\eta(Z)X-\eta(X)\eta(Z)Y)-$$
$$-\frac{(scal+6)}{2}(g(Y,Z)(\nabla_W\eta)X\xi-g(X,Z)(\nabla_W\eta)Y\xi+g(Y,Z)\eta(X)\nabla_W\xi-$$
$$-g(X,Z)\eta(Y)\nabla_W\xi+(\nabla_W\eta)Y\eta(Z)X+\eta(Y)(\nabla_W\eta)ZX-$$
$$-(\nabla_W\eta)X\eta(Z)Y-\eta(X)(\nabla_W\eta)ZY).$$
Taking $X,Y,Z,W$ orthogonal to $\xi$ and using \eqref{z1} and \eqref{e1}, we get from the above
\begin{equation}\label{e34}
(\nabla_WR)(X,Y,Z)=\frac{W(scal)}{2}(g(Y,Z)X-g(X,Z)Y)-
\end{equation}
$$-\frac{(scal+6)}{2}(-g(Y,Z)g(X,W)\xi+g(X,Z)g(Y,W)\xi).$$
From \eqref{e34} it follows that
\begin{equation}\label{e35}
\varphi^2(\nabla_WR)(X,Y,Z)=\frac{W(scal)}{2}(g(Y,Z)X-g(X,Z)Y).
\end{equation}
\end{proof}
Again, if the manifold satisfies the condition $R(X,Y).Ric=0$, then we have seen that $scal=-6$, i.e. $scal=constant$
and hence from \eqref{e45} we can state the following
\begin{thm}\label{t3}
A 3-dimensional quasi-para-Sasakian manifold $(M,\varphi,\eta,\xi,g)$ satisfying the condition $R(X,Y).Ric=0$ is locally symmetric.
\end{thm}
\begin{defn}
The Ricci tensor $Ric$ of a quasi-para-Sasakian manifold $M$ is called $\emph{$\eta-$parallel}$ if it satisfies
$(\nabla_XRic)(\varphi Y,\varphi Z)=0$ for all vector fields $X,Y$ and $Z$.
\end{defn}
The notation for Ricci-$\eta-$parallelity for Sasakian manifolds was introduce in \cite{K}.
\begin{pro}\label{pro7}
If a 3-dimensional quasi-para-Sasakian manifold $(M,\varphi,\eta,\xi,g)$ has $\eta-$ parallel Ricci tensor, then the scalar curvature $scal$ is constant.
\end{pro}
\begin{proof}
From \eqref{e27}, we get, by virtue of \eqref{con} and $\eta \circ \varphi=0$,
\begin{equation}\label{e36}
Ric(\varphi X,\varphi Y)=-\frac{(scal+2)}{2}(g(X,Y)-\eta(X)\eta(Y)).
\end{equation}
Differentiating \eqref{e36} covariantly along $Z$, we get
\begin{equation}\label{e37}
(\nabla_ZRic)(\varphi X,\varphi Y)=-\frac{Z(scal)}{2}(g(X,Y)-\eta(X)\eta(Y))+
\end{equation}
$$+\frac{(scal+2)}{2}(\eta(Y)(\nabla_Z\eta)X+\eta(X)(\nabla_Z\eta)Y).$$
By using $(\nabla_XRic)(\varphi Y,\varphi Z)=0$ and \eqref{e37}, we get
\begin{equation}\label{e38}
-Z(scal)(g(X,Y)-\eta(X)\eta(Y))+
\end{equation}
$$+(scal+2)(\eta(Y)(\nabla_Z\eta)X+\eta(X)(\nabla_Z\eta)Y)=0.$$
Taking the trace in \eqref{e38}, we get $Z(scal)=0$, for all $Z$.
\end{proof}
By virtue $Proposition~\ref{pro7}$ and  $Theorem~\ref{t2}$, we have the following
\begin{thm}\label{t3}
If a 3-dimensional quasi-para-Sasakian manifold $(M,\varphi,\eta,\xi,g)$ has $\eta-$ parallel Ricci tensor is locally $\varphi-$symmetric.
\end{thm}
\begin{thm}\label{t4}
If a 3-dimensional quasi-para-Sasakian manifold $(M,\varphi,\eta,\xi,g)$ has $\eta-$ parallel Ricci tensor, then it satisfies the condition
\begin{equation}\label{e39}
(\nabla_XRic)(Y,Z)+(\nabla_YRic)(Z,X)+(\nabla_ZRic)(X,Y)=0.
\end{equation}
\end{thm}
\begin{proof}
Taking the trace in \eqref{e38}, we obtain
\begin{equation}\label{e40}
X(scal)=0,
\end{equation}
for any vector field $X$.
From \eqref{e27}, we have
\begin{equation}\label{e41}
(\nabla_ZRic)(X,Y)=\frac{Z(scal)}{2}(g(X,Y)-\eta(X)\eta(Y))-
\end{equation}
$$-\frac{(scal+6)}{2}(\eta(Y)(\nabla_Z\eta)X+\eta(X)(\nabla_Z\eta)Y).$$
Now using \eqref{e40} and \eqref{e41}, we have
\begin{equation}\label{e42}
(\nabla_ZRic)(X,Y)=-\frac{(scal+6)}{2}(\eta(Y)(\nabla_Z\eta)X+\eta(X)(\nabla_Z\eta)Y)
\end{equation}
By virtue of \eqref{e42}, we get from \eqref{e1} that
$$(\nabla_XRic)(Y,Z)+(\nabla_YRic)(Z,X)+(\nabla_ZRic)(X,Y)=0.$$
\end{proof}

Unlike the case when the dimension is greater than 3, when the dimension is equal to 3 the right-hand parenthesis in \eqref{e14} is trivially equal to 0 and thus nothing follows (from \eqref{e14}) about the scalar curvature. We shall give an example of a 3-dimensional quasi-para-Sasakian manifold, which has scalar curvature equal to $-6$.

\begin{example}
Let $L$ be a 3-dimensional real connected Lie group and ${\g}$ be its Lie algebra with a basis $\{E_1,E_2,E_3\}$ of left invariant vector fields (see \cite{ZN},
by the following commutators:
\begin{equation}
[E_1,E_2]=2E_3, \quad [E_1,E_3]=2E_2, \quad [E_2,E_3]=2E_1.
\end{equation}

We define an almost paracontact structure $(\varphi ,\xi ,\eta)$ and a pseudo-Riemannian metric $g$ in the following way:
\[
\begin{array}{llll}
\varphi E_1=E_2 , \quad \varphi E_2=E_1 , \quad \varphi E_3=0 \\
\xi =E_3 , \quad \eta (E_3)=1 , \quad \eta (E_1)=\eta (E_2)=0 , \\
g(E_1,E_1)=g(E_3,E_3)=-g(E_2,E_2)=1 ,\\
\quad g(E_i,E_j)=0, \quad i\neq j \in \{1,2,3\}.
\end{array}
\]
Then $(L,\varphi ,\xi ,\eta ,g)$ is a 3-dimensional almost paracontact manifold. Since the metric $g$ is left invariant, the Koszul equality becomes
\[
\begin{array}{llll}
\nabla_{E_1}E_1=0 , \quad \nabla_{E_1}E_2=E_3 , \quad \nabla_{E_1}E_3=E_2 , \\
\nabla_{E_2}E_1=-E_3 , \quad \nabla_{E_2}E_2=0, \quad \nabla_{E_2}E_3=E_1  , \\
\nabla_{E_3}E_1=-E_2 , \quad \nabla_{E_3}E_2=-E_1 , \quad \nabla_{E_3}E_3=0.   \\
\end{array}
\]
It is not hard to see that the Ricci tensor $Ric$ is equal to
$$Ric(X,Y)=\frac{scal}{3}g(X,Y),$$
where $scal=-6$.
\end{example}

\section*{Acknowledgments}

S.Z. is partially supported by Contract DN 12/3/12.12.2017 and Contract 80-10-33/2017 and its sequel with the Sofia University "St.Kl.Ohridski".

\end{document}